\RequirePackage{amsmath}
\documentclass{llncs}

\usepackage[utf8]{inputenc}
\usepackage{latexsym}
\usepackage[T1]{fontenc}
\usepackage{amssymb}
\usepackage{stmaryrd}
\usepackage{marvosym}
\usepackage{mdwmath}
\usepackage{mdwtab}

\usepackage{graphicx}
\usepackage[cmyk]{xcolor}
\usepackage{tikz}

\usepackage{tabularx}
\usepackage{array}
\usepackage{eqparbox}
\usepackage{booktabs}

\usepackage{url}
\usepackage{hyperref}

\newcommand{\pc}{\mathbf{P}}
\newcommand{\oc}{\mathbf{O}}

\newcommand{\pp}{\mathbb{P}}

\newcommand{\ideal}{$\sigma$-ideal}

\begin{document}

\setcounter{page}{1}
 
\title{Generalized Ideals and Co-Granular Rough Sets }

\author{\textbf{A. Mani}}
\institute{Department of Pure Mathematics\\
University of Calcutta\\
9/1B, Jatin Bagchi Road\\
Kolkata(Calcutta)-700029, India\\
\email{$a.mani.cms@gmail.com$}\\
Homepage: \url{http://www.logicamani.in}}

\maketitle

\begin{abstract}
Lattice-theoretic ideals have been used to define and generate non granular rough approximations over general approximation spaces over the last few years by few authors. The goal of these studies, in relation based rough sets, have been to obtain nice properties comparable to those of classical rough approximations. In this research paper, these ideas are generalized in a severe way by the present author and associated semantic features are investigated by her. Granules are used in the construction of approximations in implicit ways and so a concept of co-granularity is introduced. Knowledge interpretation associable with the approaches is also investigated. This research will be of relevance for a number of logico-algebraic approaches to rough sets that proceed from point-wise definitions of approximations and also for using alternative approximations in spatial mereological contexts involving actual contact relations. The antichain based semantics invented in earlier papers by the present author also applies to the contexts considered. 

\end{abstract}

\section{Introduction}

In general rough set theory that specifically targets information systems as the object of study, approximations are defined relative to information systems and rough objects of various kinds are studied \cite{AM240,ppm2,YL96}. These approximations may be defined relative to some concept of granules or they may be defined without direct reference to any concept of granules or granulations. The corresponding approximations are in general not equivalent. Among the latter class, few new approximations have been studied in \cite{AT2011,ABA2008,KYG2016} over general approximation spaces of the form $(X, R)$ with $X$ being a set and $R$ being at least a reflexive relation. \emph{In these approximations, a point is in an approximation of a subset of $X$ if it satisfies a condition that involves lattice ideals of the Boolean power set lattice}. The formalism in the overview paper \cite{KYG2016} makes use of a more laborious formalism - but is essentially equivalent to what has been stated in the last sentence. The significance of the obtained results and potential application contexts are not explored in the three papers mentioned \cite{AT2011,ABA2008,KYG2016} in sufficient detail and many open problems remain hidden.

In this research paper, a definition of co-granularity of approximations is introduced, the methodology is generalized to specific modifications of granular operator spaces \cite{AM6999,AM9114} (called co-granular operator spaces) and in particular to lattices generated by collections of sets and lattice ideals, connections with general approximation spaces (or adjacency spaces) are dropped, connections with granular operator spaces are established, issues in possible semantics of the generalized approach are computed, knowledge interpretation in the contexts are proposed, meaningful examples are constructed,ideal based rough approximations are shown to be natural in spatial mereological contexts and related problems are posed.  One meaning that stands out in all this is that \emph{if a property has little to do with what something is not, then that something has the property in an approximate sense}. This idea might work in some contexts - the developed/invented formalisms suggest some restrictions on possible contexts.

Ideals and filters have been used by the present author in algebraic semantics of general rough sets in some of her earlier papers like \cite{AM105,AM240,AM3600,AM9501}. Concepts of rough ideals have also been studied by different authors in specific algebras (see for example \cite{QQL2014,EHB2012})- these studies involve the use of rough concepts within algebras. The methodology of the present paper does not correspond to those used in the mentioned papers in a direct way.
   
In the next section, some of the essential background is mentioned. In the third section, generalized set theoretic frameworks are introduced and properties of approximations are proved. In the following section, co-granular operator spaces are defined and studied. In the fifth section, the meaning of the approximations and generalizations are explained for the first time and both abstract and concrete examples are constructed. Mereotopological approximations are invented/developed over very recent work on actual contact algebras in the sixth section.

\section{Background}

By an \emph{Information System} $\mathcal{I}$, is meant a tuple of the form \[\mathcal{I}\,=\, \left\langle \mathbb{O},\, At,\, \{V_{a} :\, a\in At\},\, \{f_{a} :\, a\in At\}  \right\rangle \]
with $\mathbb{O}$, $At$ and $V_{a}$ being respectively sets of \emph{Objects}, \emph{Attributes} and \emph{Values} respectively. In general the valuation has the form $f_{a}:\, O\,\longmapsto \, \wp(V)$, where $V\,=\, \bigcup V_{a}$ (as in indeterminate information systems). These can be used to generate various types of relational, covering or relator spaces which in turn relate to approximations of different types and form a substantial part of the problems encountered in general rough set theories. One way of defining an indiscernibility relation $\sigma$ is as below:

For $x,\, y\,\in\, \mathbb{O} $ and $B\,\subseteq\, At $, $(x,\,y)\,\in\, \sigma $ if and only if $(\forall a\in B)\, \nu(a,\,x)\,=\, \nu (a,\, y) $. In this case $\sigma$ is an equivalence relation. Lower and upper approximations, rough equalities are defined over it and topological algebraic semantics can be formulated over \emph{roughly equivalent} objects (or subsets of attributes) through extra operations. Duality theorems, proved for \emph{pre-rough} algebras defined in \cite{BC1}, are specifically for structures relation isomorphic to the approximation space $\left(\mathbb{O}, \sigma \right)$. This is also true of the representation results in \cite{AM3,du,IE2013}. But these are not for information systems - optimal concepts of \emph{isomorphic information systems} are considered by the present author in a forthcoming paper.

In fact in \cite{AM3}, it has been proved by the present author that 
\begin{theorem}
For every super rough algebra $S$, there exists an approximation space $X$ such that the super rough set algebra generated by $X$ is isomorphic to $S$.
\end{theorem}

The concept of \emph{inverse problem} was introduced by the present author in \cite{AM3} and was subsequently refined in \cite{AM240}. In simple terms, the problem is a generalization of the duality problem which may be obtained by replacing the semantic structures with parts thereof. Thus the goal of the problem is to fit a given set of approximations and some semantics to a suitable rough process originating from an information system. From the practical perspective, a few (as opposed to one) theoretical frameworks may be suitable for formulating the problem itself. The theorem mentioned above is an example of a solution of the \emph{inverse problem} in the associated context.

The definition of approximations maybe granular, point wise, abstract or otherwise. 
In simple terms, \emph{granules} are the subsets (or objects) that generate approximations and \emph{granulations} are the collections of all such granules in the context. For more details see \cite{AM240,AM3930}. In this paper a variation of generalized granular operator spaces, introduced and studied by the present author in \cite{AM6999,AM9114}, will serve as the main framework for most considerations. For reference, related definitions are mentioned below.

\begin{definition}\label{gos}
A \emph{Granular Operator Space}\cite{AM6999} \textsf{GOS} $S$ is a structure of the form $S\,=\, \left\langle \underline{S}, \mathcal{G}, l , u\right\rangle$ with $\underline{S}$ being a set, $\mathcal{G}$ an \emph{admissible granulation}(defined below) over $S$ and $l, u$ being operators $:\wp(\underline{S})\longmapsto \wp(\underline{S})$ ($\wp(\underline{S})$ denotes the power set of $\underline{S}$) satisfying the following ($\underline{S}$ is replaced with $S$ if clear from the context. \textsf{Lower and upper case alphabets may denote subsets} ):

\begin{align*}
a^l \subseteq a\,\&\,a^{ll} = a^l \,\&\, a^{u} \subseteq a^{uu}  \\
(a\subseteq b \longrightarrow a^l \subseteq b^l \,\&\,a^u \subseteq b^u)\\
\emptyset^l = \emptyset \,\&\,\emptyset^u = \emptyset \,\&\,\underline{S}^{l}\subseteq S \,\&\, \underline{S}^{u}\subseteq S.
\end{align*}

In the context of this definition, \emph{Admissible Granulations} are granulations $\mathcal{G}$ that satisfy the following three conditions ($t$ is a term operation formed from the set operations $\cup, \cap, ^c, 1, \emptyset$):

\begin{align*}
(\forall a \exists
b_{1},\ldots b_{r}\in \mathcal{G})\, t(b_{1},\,b_{2}, \ldots \,b_{r})=a^{l} \\
\tag{Weak RA, WRA} \mathrm{and}\: (\forall a)\,(\exists
b_{1},\,\ldots\,b_{r}\in \mathcal{G})\,t(b_{1},\,b_{2}, \ldots \,b_{r}) =
a^{u},\\
\tag{Lower Stability, LS}{(\forall b \in
\mathcal{G})(\forall {a\in \wp(\underline{S}) })\, ( b\subseteq a\,\longrightarrow\, b \subseteq a^{l}),}\\
\tag{Full Underlap, FU}{(\forall
a,\,b\in\mathcal{G})(\exists
z\in \wp(\underline{S}) )\, a\subset z,\,b \subset z\,\&\,z^{l} = z^{u} = z,}
\end{align*}
\end{definition}

The concept of admissible granulation was defined for rough Y-systems \textsf{RYS} (a more general framework due to the present author in \cite{AM240}) using parthoods instead of set inclusion and relative to \textsf{RYS}, $\pc = \subseteq$, $\pp = \subset$ in granular operator spaces \cite{AM6999}. 
The concept of \emph{generalized granular operator spaces} has been introduced in \cite{AM9114} by the present author as a proper generalization of that of granular operator spaces. The main difference is in the replacement of $\subset$ by arbitrary \emph{part of} ($\pc$) relations in the axioms of admissible granules and inclusion of $\pc$ in the signature of the structure. 

\begin{definition}
A \emph{General Granular Operator Space} (\textsf{GSP}) $S$ is a structure of the form $S\,=\, \left\langle \underline{S}, \mathcal{G}, l , u, \pc \right\rangle$ with $\underline{S}$ being a set, $\mathcal{G}$ an \emph{admissible granulation}(defined below) over $S$, $l, u$ being operators $:\wp(\underline{S})\longmapsto \wp(\underline{S})$ and $\pc$ being a definable binary generalized transitive predicate (for parthood) on $\wp(\underline{S})$ satisfying the same conditions as in Def.\ref{gos} except for those on admissible granulations (Generalized transitivity can be any proper nontrivial generalization of parthood (see \cite{AM9501}). $\pp$ is  proper parthood (defined via $\pp ab$ iff $\pc ab \,\&\,\neg \pc ba$) and $t$ is a term operation formed from set operations on the powerset $\wp(S)$):

\begin{align*}
(\forall x \exists
y_{1},\ldots y_{r}\in \mathcal{G})\, t(y_{1},\,y_{2}, \ldots \,y_{r})=x^{l} \\
\tag{Weak RA, WRA} \mathrm{and}\: (\forall x)\,(\exists
y_{1},\,\ldots\,y_{r}\in \mathcal{G})\,t(y_{1},\,y_{2}, \ldots \,y_{r}) =
x^{u},\\
\tag{Lower Stability, LS}{(\forall y \in
\mathcal{G})(\forall {x\in \wp(\underline{S}) })\, ( \pc yx\,\longrightarrow\, \pc yx^{l}),}\\
\tag{Full Underlap, FU}{(\forall
x,\,y\in\mathcal{G})(\exists
z\in \wp(\underline{S}) )\, \pp xz,\,\&\,\pp yz\,\&\,z^{l} = z^{u} = z,}
\end{align*}
\end{definition}

There are ways of defining rough approximations that do not fit into the above frameworks and the present paper is mainly about specific such cases.

\subsection{Ideals on Posets}\label{wth}

A \emph{lattice ideal} $K$ of a lattice $L= (L, \vee, \wedge )$ is a subset of $L$ that satisfies the following ($\leq$ is assumed to the definable lattice order on $L$):
\begin{align}
(\forall a \in L)(\forall b\in K)(a\leq b \longrightarrow a\in K) \tag{o-Ideal}\\
(\forall a, b\in K)\, a\vee b \in K   \tag{Join Closure}
\end{align}

An ideal $P$ in a lattice $L$ is \emph{prime} if and only if $(\forall a, b )(a\wedge b\in P \longrightarrow a \in P \text{ or } b\in P)$. $Spec(L)$ shall denote the set of all prime ideals. Maximal lattice filters are the same as ultrafilters. In Boolean algebras, any filter $F$ that satisfies $(\forall a) a\in F \text{ or } a^c \in F$ is an ultra filter.  \emph{Chains} are subsets of a poset in which any two elements are comparable, while \emph{antichains} are subsets of a poset in which no two distinct elements are comparable. Singletons are both chains and antichains.

\subsection{Ideal Based Framework}

The approximations in \cite{KYG2016} are more general than the ones introduced and studied in \cite{AT2011,ABA2008}. A complete reformulation of the main definition and approximation is presented in this subsection. These approximations are not granular in any obvious way and need not fit into generalized granular operator spaces.

\begin{definition}
\begin{itemize}
\item {Let $\langle X, R \rangle$  be a general approximation space with $X$ being a set and $R$ being a reflexive binary relation on $X$,}
\item {$\langle \wp (X), \cup, \cap, \emptyset, X \rangle $ be the Boolean lattice on the power set of $X$. Any lattice ideal in the Boolean lattice would be referred to as an \emph{ideal} and the collection of all ideals would be denoted by $\mathcal{I}(X)$ (this is algebraic distributive lattice ordered and the implicit $\sigma$ used is the Boolean order $\subseteq$ ) and $\mathbb{I}$ a fixed ideal in it.}
\item {Let $[x]^R = \{a : a \in X \,\&\, Rxa \}$ and $<x> = \bigcap \{ [b]^R:\, b\in X \,\&\,{x\in [b]^R}  \}$. }
\item {$(\forall A\in \wp(X))\, A^{l_\kappa} = \{a : a \in A\, \&\, <a> \cap A^c \in \mathbb{I} \}$}
\item {$(\forall A\in \wp(X))\, A^{u_\kappa} = \{a : a \in X\, \&\,<a>\cap A \notin \mathbb{I}\} \cup A$ }
\end{itemize}
\end{definition}

The approximations have properties similar to those of approximations in classical rough set theory using the point wise definition of approximations. This is mainly due to the nature of the sets of the form $<a>$.

\begin{theorem}
All of the following hold for any subset $A, B\subseteq  X$ in the context of the above definition:
\begin{itemize}
\item {$A^{l_\kappa} \subseteq A \subseteq A^{u_\kappa} $ and $\emptyset^{u_\kappa}=\emptyset , X^{l_\kappa} = X $}
\item {$A\subset B\longrightarrow A^{l_\kappa}\subseteq B^{l_\kappa} \,\&\,A^{u_\kappa}\subseteq B^{u_\kappa} $}
\item {$A^{{l_\kappa}{l_\kappa}} = A^{l_\kappa}$, $A^{{u_\kappa}{u_\kappa}} = A^{u_\kappa}$; $(\forall A\in \mathbb{I})\, A^{u_\kappa} = A$}
\item {$(A\cap B)^{l_\kappa} = A^{l_\kappa}\cap B^{l_\kappa}$ and $(A\cup B)^{l_\kappa} \supseteq A^{l_\kappa}\cup B^{l_\kappa}$ }
\item {$(A\cup B)^{u_\kappa} = A^{u_\kappa}\cup B^{u_\kappa}$ and $(A\cap B)^{u_\kappa} \subseteq A^{u_\kappa}\cap B^{u_\kappa}$ }
\item {$A^{u_\kappa}=(A^{c})^{l_\kappa c}$, $A^{l_\kappa}=(A^{c})^{u_\kappa c}$ and $(\forall A^c\in \mathbb{I})\,  A^{l_\kappa} = A$.}
\end{itemize}
$\tau_R^* = \{A : A^{l_\kappa} = A\}$ is a topology.
\end{theorem}

\section{Set-Theoretic Generalization of Ideal-Based Framework }\label{sei}

Set theoretic generalizations of the approach in \cite{AT2011,ABA2008,KYG2016} are proposed in this section by the present author. 

If $R$ is a binary relation on a set $S$, then for any $x\in S$, the \emph{successor neighborhood } $[x]_R$  generated by $x$ is the set $[x]_R = \{a : \, a\in S \, \&\, Rax\}$, while the \emph{predecessor neighborhood} $[x]^R$ is the set $[x]^R = \{a : \, a\in S \, \&\, Rxa\}$. 

\begin{definition}
\begin{itemize}
\item {Let $\langle X, R \rangle$  be a general approximation space with $X$ being a set and $R$ being a reflexive binary relation on $X$,}
\item {$\mathcal{X}$ be a distributive lattice (a ring of subsets of $X$). Let $\mathcal{I}(X)$ be the lattice of lattice ideals of $\mathcal{X}$ and $\mathbb{I}\in \mathcal{I}(X)$ }
\item {Let $[x]^R = \{a : a \in X \,\&\, Rxa \}$ and $<x> = \bigcap \{ [b]^R:\, b\in X \,\&\,{x\in [b]^R}  \}$. }
\item {$(\forall A\in \wp(X))\, A^{l_k} = \{a : a \in A\, \&\, <a> \cap A^c \in \mathbb{I} \}$}
\item {$(\forall A\in \wp(X))\, A^{u_k} = \{a : a \in X\, \&\,<a>\cap A \notin \mathbb{I} \} \cup A$ }
\item {If $\mathcal{I} (X)$ is replaced by $Spec(X)$ in the last two statements, then the resulting lower and upper approximations will be denoted respectively by $l_p$ and $u_p$.}
\end{itemize}
The approximations will be referred to as \emph{Distributive set approximations by ideals} (IAD approximations). If in the second condition if $\mathcal{X}$ is an algebra of subsets of $X$ instead, then the definitions of the approximations can be improved as below: 
\begin{itemize}
\item {$(\forall A\in \mathcal{X})\, A^{l_+} = \{a : a \in A\, \&\, <a> \setminus A \in \mathbb{I} \}$}
\item {$(\forall A\in \mathcal{X})\, A^{u_+} = \{a : a \in X\, \&\,<a>\cap A \notin \mathbb{I} \} \cup A$ }
\end{itemize}
These approximations will be referred to as \emph{Set difference approximations by ideals} (IASD approximations). 
\end{definition}

When the ideals refer to a ring of subsets, the operations used in the definition of IAD approximations refer to an algebra of sets over $X$. IASD approximations are better behaved. 

\begin{proposition}
In the above, all of the following hold:
\begin{itemize}
\item {For $a, b\in X $, if $a\in <b>$ then $<a> \subseteq <b>$.}
\item {If $\tau ab$ if and only $a\in <b>$ then $\tau$ is a reflexive, transitive and weakly antisymmetric relation in the sense, if $\tau ab \,\&\, \tau ba$ then $<a>=<b>$. }
\end{itemize}
\end{proposition}

\begin{theorem}
The IAD approximations are well defined and satisfy all of the following for any subsets $A$ and $B$:
\begin{align}
A^{l_k} \subseteq A \subseteq A^{u_k}\\
\emptyset^{u_k}=\emptyset \,;\, X^{l_k} = X \\
A\subset B\longrightarrow A^{l_k}\subseteq B^{l_k} \,\&\,A^{u_k}\subseteq B^{u_k} \\ 
A^{l_k l_k} = A^{l_k} ; \; A^{u_k u_k} = A^{u_k}\\
(A\cap B)^{l_k} = A^{l_k}\cap B^{l_k} ; \; (A\cup B)^{l_k}\supseteq A^{l_k}\cup B^{l_k}\\
(A\cup B)^{u_k} = A^{u_k}\cup B^{u_k} ; \; (A\cap B)^{u_k} \subseteq A^{u_k}\cap B^{u_k}
\end{align}
\end{theorem}
\begin{proof}
\begin{itemize}
\item {$A^{l_k} \subseteq A \subseteq A^{u_k} $ follows from definition}
\item {$\emptyset^{u_k}=\emptyset$ because it contains no elements. $X^{l_k} = X $ because $ <x> \cap \emptyset = \emptyset$ is a trivial ideal.}
\item {\begin{align*}
\text{Let } A\subset B \\
\text{If } x\in A^{u_k} \text{ then } <x> \cap A \notin \mathbb{I}\\
\text{So } <x>\cap B\notin \mathbb{I} \text{ and } x\in B^{u_k} \\
\text{If } x\in A^{l_k} \text{ then } <x> \cap A^c \in \mathbb{I}\\
<x>\cap B^c \subset <x>\cap A^c \in \mathbb{I}\\
\text{ and so } <x>\cap B^c \in \mathbb{I}(X) \text{ and } x\in B^{l_k}.
\end{align*}
}
\item {If $a\in A^{l_k}$ then $<a>\cap A^c \in \mathbb{I} $ and 
$<a>\cap A^c \subseteq\, <a>\cap A^{l_k c}$. The converse also holds because of the definition of $<a>$. 
So $A^{l_k l_k}= A^{l_k}$.}
\item {It is obvious that $A^{u_k}\subseteq A^{u_k u_k}$. If $x\in A^{u_k u_k}$, then $x\in A $ or $x\in <x>\cap A^{u_k}\notin \mathbb{I}$. As $<x> \cap A \subseteq <x> \cap A^{u_k}$,  so $A^{u_k u_k} = A^{u_k}$.}
\item { $x \in (A\cap B)^{l_k}$, if and only if $<x>\cap(A\cap B)^c \in \mathbb{I}$ and $x\in A\cap B$
if and only if $(<x>\cap A^c\in \mathbb{I}$ and $<x>\cap B^c\in \mathbb{I}$ and $x\in A\cap B$
So $(A\cap B)^{l_k} = A^{l_k}\cap B^{l_k}$.}
\item {$x\in (A\cup B)^{l_k}$ if and only if $<x> \cap (A\cup B)^c \in \mathbb{I}$ and $x\in A\cup B$ if and only if 
$(<x> \cap A^c)\cap (<x>\cap B^c) \in \mathbb{I}$ and $x\in A\cup B$. This implies 
$(<x> \cap A^c) \in \mathbb{I}$ and $x\in A$ or $<x>\cap B^c \in \mathbb{I}$ and $x\in B$.
So $(A\cup B)^{l_k}\supseteq A^{l_k}\cup B^{l_k}$. }
\item { $x\in (A\cup B)^{u_k}$ if and only if $x\in A\cup B$ or $<x>\cap (A\cup B) \notin \mathbb{I}$.
If $x\in A\cup B$, then $x\in A^{u_k}\cup B^{u_k}$. $<x>\cap (A\cup B) \notin \mathbb{I}$ if and only if $<x>\cap A\notin \mathbb{I}$ and $<x>\cap B \notin \mathbb{I}$. So it follows that $x\in A^{u_k}\cup B^{u_k}$ and conversely.}
\item {$x\in (A\cap B)^{u_k}$ if and only if $x\in A\cap B$ or $<x>\cap (A\cap B) \notin \mathbb{I}$.
If $x\in A\cap B$, then $x\in A^{u_k}\cap B^{u_k}$. $<x>\cap (A\cap B) \notin \mathbb{I}$ if and only if $<x>\cap A\notin \mathbb{I}$ or $<x>\cap B \notin \mathbb{I}$. So it follows that $(A\cap B)^{u_k} \subseteq A^{u_k}\cap B^{u_k}$.}
\end{itemize}
\qed
\end{proof}

Actually full complementation can be omitted and replaced with set difference. This way the approximations can be defined on subsets of the powerset.

\begin{theorem}
The IASD approximations are well defined and satisfy all of the following for any subsets $A$ and $B$ in $\mathcal{X}$:
\begin{align}
A^{l_+} \subseteq A \subseteq A^{u_+} ; \; \emptyset^{u_+}=\emptyset \,;\, X^{l_+} = X \\
A\subset B\longrightarrow A^{l_+}\subseteq B^{l_+} \,\&\,A^{u_+}\subseteq B^{u_+} \\ 
A^{l_+ l_+} = A^{l_+} ; \; A^{u_+ u_+} = A^{u_+}\\
(A\cap B)^{l_+} = A^{l_+}\cap B^{l_+} ; \; (A\cup B)^{l_+}\supseteq A^{l_+}\cup B^{l_+}\\
(A\cup B)^{u_+} = A^{u_+}\cup B^{u_+} ; \; (A\cap B)^{u_+} \subseteq A^{u_+}\cap B^{u_+}
\end{align}
\end{theorem}

\begin{proof}
 The proof is similar to that of the previous theorem. 
 Relative complementation suffices.
\qed
\end{proof}

\begin{remark}
The main advantages of the generalization are that knowledge of complementation is not required in construction of the IASD approximations, a potentially restricted collection of ideals is usable in the definition of approximations and this in turn improves computational efficiency.
\end{remark}

\section{Co-Granular Operator Spaces By Ideals}

Given a binary relation on a set it is possible to regard specific subsets of the set as generalized ideals relative to the relation in question. The original motivations for the approach relate to the strategies for generalizing the concept of lattice ideal to partially ordered sets (see \cite{JC1977,SR2015,PVV1971}).  It is also possible to use a binary relation on the power set to form generalized ideals consisting of some subsets of the set. Both approaches are apparently compatible with the methods used for defining approximations by ideals.

In this section the two possibilities are examined and generalizations called \emph{co-granular operator spaces by ideals} and \emph{higher co-granular operator spaces} are proposed.

\begin{definition}\label{fra}
Let $H$ be a set  and $\sigma$ a binary relation on $H$ (that is $\sigma \subseteq H^2$) then 
\begin{itemize}
\item {The \emph{Principal Up-set} generated by $a, b\in H$ shall be the set \[U(a, b) = \{x: \sigma ax \,\&\,\sigma b x \}.\] } 
\item {The \emph{Principal Down-set} generated by $a, b\in H$ shall be the set\[ L(a, b) = \{x: \sigma xa \,\&\,\sigma xb \}. \]}
\item {$B\subseteq H$ is \emph{U-directed} if and only if $(\forall a, b\in B)\, U(a, b)\cap B \neq \emptyset$.}
\item {$B\subseteq H$ is \emph{L-directed} if and only if $(\forall a, b\in B)\, L(a, b)\cap B \neq \emptyset$. If $B$ is both U- and L-directed, then it is $\sigma$-\emph{directed}.}
\item {$K\subset H$ is a $\sigma$-\emph{ideal} if and only if 
\begin{align}
(\forall x\in H)(\forall a\in K)(\sigma xa \longrightarrow x\in K)\\
(\forall a, b\in K)\, U(a, b) \cap K \neq \emptyset 
\end{align}}
\item {$F\subset H$ is a $\sigma$-\emph{filter} if and only if 
\begin{align}
(\forall x\in H)(\forall a\in F)(\sigma ax \longrightarrow x\in F)\\
(\forall a, b\in F)\, L(a, b) \cap F \neq \emptyset 
\end{align}}
\item {The set of $\sigma$-ideals and $\sigma$-filters will respectively be denoted by $\mathcal{I}(H)$ and $\mathcal{F}(H)$ respectively. These are all partially ordered by the set inclusion order. If the intersection of all \ideal s containing a subset $B\subset H$ is an \ideal , then it will be called the {\ideal}  generated by $B$ and denoted by $\langle B \rangle$. The collection of all principal \ideal s will be denoted by $\mathcal{I}_p (H)$. If $\langle x \rangle$ exists for every $x\in H$, then $H$ is said to be $\sigma$-\emph{principal} (principal for short). }
\item {A {\ideal}  $K$ will be said to be \emph{prime} if and only if 
\[(\forall a, b\in H)(L(a, b) \cap K \neq \emptyset \longrightarrow a\in K \text{ or } b\in K).\] The dual concept for filters can also be defined.}
\item {A subset $B\subseteq H$ will be said to be $\sigma$-\emph{convex} if and only if 
\[(\forall a, b\in B)(\forall x\in H)(\sigma ax \,\&\,\sigma x b \longrightarrow x\in B )\] }
\end{itemize}
\end{definition}

\begin{proposition}
All of the following hold in the context of the above definition:
\begin{itemize}
\item {All {\ideal s} are $\sigma$-convex and U-directed.}
\item {If $H$ is $\sigma$-directed, then all \ideal s are $\sigma$-directed subsets.}
\item {Every {\ideal} is contained in a maximal {\ideal}.}
\item {If $H$ is L-directed, $K$ is a prime ideal and for $K_1, K_2 \in \mathcal{I}$ if $K_1\cap K_2 \subseteq K$, then $K_1 \subseteq K_2$ or $K_2 \subseteq K_1$.}
\item {If $\langle a \rangle,\, \langle b\rangle \in \mathcal{I}_p (H) $ and $\tau(\sigma)ab$ ($\tau(\sigma)$ being the transitive completion of $\sigma$), then $\langle a \rangle\,\subseteq \, \langle b\rangle$.}
\end{itemize}
\end{proposition}
\begin{proof}
The proofs are not too complex and may be found in \cite{JC1977}  
\end{proof}

\begin{proposition}
Neighborhoods generated by points relate to bound operators according to $(\forall x) [x]_\sigma = L_\sigma (x, x) \;\&\; [x]^\sigma = U_\sigma (x, x)$.
\end{proposition}

\begin{remark}
The connection between the two is relevant when $\sigma$ is used for generating ideals and also for the neighborhoods. There are no instances of such usage in the literature as of this writing and is an open area for further investigation. 
\end{remark}

\begin{definition}
In the context of Def.\ref{fra}, $\sigma$ will be said to be \emph{supremal} if and only if 
\begin{equation}
(\forall a, b\in H)(\exists !^{>0} s(a, b) \in U(a, b))(x\in U(a, b)\longrightarrow s(a, b) = x \text{ or } \sigma s(a, b) x) 
\end{equation}
Elements of the form $s(a, b)$ are $\sigma$-\emph{supremums} of $a$ and $b$. 
\end{definition}

\begin{theorem}
All of the following hold:
\begin{itemize}
\item {Anti symmetrical relations are uniquely supremal.}
\item {\ideal s are closed under supremal relations.}
\item {If $\sigma$ is supremal then $\langle K \rangle $ exists for all nonempty subsets $K\subseteq H$ in $(H, \sigma)$ and is principal.}
\item {If $\sigma$ is supremal and $(\mathcal{I}(H), \subseteq)$ has a least element, then it is an algebraic lattice and the finitely generated \ideal s are its compact elements. }
\item {If $\sigma$ is supremal, let $\mathfrak{L, \lambda,} \pi, \Sigma: \wp(H)\setminus \{\emptyset\} \longmapsto \wp (H)$ be maps such that for any $\emptyset \neq X \subseteq H$, 
\begin{align}
\mathfrak{L}(X) = \{x \in H; \exists a\in X\, \sigma x a\} \text{ and } \lambda (X) = \mathfrak{L}(X) \cup X\\ 
\pi (X) = \{a \in H ; (\exists b, c \in X)\, a= s(b, c)\} \text{ and } \Sigma (X) = \pi(X) \cup X,
\end{align} then $\langle X \rangle= \bigcup_1^\infty (\Sigma\lambda)^n (X)$. If $\sigma$ is also reflexive, then $\langle X \rangle= \bigcup_1^\infty (\pi\mathfrak{L})^n (X)$. }
\item {If $(\forall a, b)\, \sigma a b \text{ or } \sigma b a$, then $(H, \sigma )$ is principal, $(\forall a \in H) \langle a \rangle = \{x: \, \tau(\sigma) x a \}$ and $(\mathcal{I}(H), \subseteq)$ is a chain.}
\item {$S$ is principal and for each $a\in S$, $\langle a \rangle = \{b ; \, \sigma ba \}$ if and only if $\sigma$ is a quasi order.}
\end{itemize}
\end{theorem}

The above results mean that very few assumptions on $\sigma$ suffice for reasonable properties on $\mathcal{I}(H)$. 

\begin{definition}
By a neighborhood granulation $\mathcal{G}$ on a set $S$ will be meant a subset of the power set $\wp(S)$ for which there exists a map $\gamma : S\longmapsto \mathcal{G} $ such that 
\begin{align}
(\forall B\in \mathcal{G})(\exists x \in S) \, \gamma (x) = B \tag{Surjectivity}\\
\bigcup_{x\in S} \gamma (x) = S   \tag{Cover}
\end{align}
\end{definition}

\begin{definition}
By a \emph{Co-Granular Operator Space By Ideals} \textsf{GOSI} will be meant a structure of the form $S\,=\, \left\langle \underline{S}, \sigma, \mathcal{G}, l_* , u_*\right\rangle$ with $\underline{S}$ being a set, $\sigma$ being a binary relation on $S$, $\mathcal{G}$ a \emph{neighborhood granulation} over $S$ and $l_*, u_*$ being \emph{*-lower and *-upper approximation operators} $:\wp(\underline{S})\longmapsto \wp(\underline{S})$ ($\wp(\underline{S})$ denotes the power set of $\underline{S}$) defined as below ($\underline{S}$ is replaced with $S$ if clear from the context. \textsf{Lower and upper case alphabets may denote subsets} ):

\begin{align*}
(\forall X\in \wp(S))\, X^{l_*} = \{a : a \in X\, \&\, \gamma(a) \cap X^c \in \mathcal{I}_\sigma (S) \} \tag{*-Lower}\\
(\forall X\in \wp(S))\, X^{u_*} = \{a : a \in S\, \&\,\gamma(a)\cap X \notin \mathcal{I}_\sigma (S)\} \cup X \tag{*-Upper}\\
\end{align*}
 
In general, if rough approximations are defined by expressions of the form $X^{\oplus} = \{a: \gamma(a)\odot X^*\in \mathcal{J}   \} $ with $\oplus\in \{l, u\}$, $\mathcal{G}\subset \wp(S)$, $\gamma: S \longmapsto \mathcal{G}$ being a map, $* \in \{c, 1\}$ and $\odot \in \{\cap , \cup\}$ , then the approximation will be said to be \emph{co-granular}. 
\end{definition}
The definition of co-granularity can be improved/generalized in a first order language with quantifiers.
\begin{definition}
By a \emph{Higher Co-Granular Operator Space By Ideals} \textsf{GOSIH} will be meant a structure of the form $S\,=\, \left\langle \underline{S}, \sigma, \mathcal{G}, l_o , u_o \right\rangle$ with $\underline{S}$ being a set, $\sigma$ \emph{being a binary relation on the powerset} $\wp(S)$, $\mathcal{G}$ a \emph{neighborhood granulation} over $S$ and $l_o, u_o$ \emph{o-lower and o-upper approximation operators} $:\wp(\underline{S})\longmapsto \wp(\underline{S})$ ($\wp(\underline{S})$ denotes the power set of $\underline{S}$) defined by the following conditions($\underline{S}$ is replaced with $S$ if clear from the context. \textsf{Lower and upper case alphabets may denote subsets} ):
\begin{align*}
\text{For a fixed } \mathbb{I} \in \mathcal{I}_\sigma (\wp(S)) \tag{Ideal}\\
(\forall X\in \wp(S))\, X^{l_o} = \{a : a \in X\, \&\, \gamma(a) \cap X^c \in \mathbb{I} \} \tag{o-Lower}\\
(\forall X\in \wp(S))\, X^{u_o} = \{a : a \in S\, \&\,\gamma(a)\cap X \notin \mathbb{I}\} \cup X \tag{o-Upper}\\
\end{align*}
\end{definition}

\begin{theorem}
All of the following hold in a GOSI $S$:
\begin{align}
(\forall A \in \wp(S))\, A^{l_*} \subseteq A \subseteq A^{u_*} ) \tag{Inclusion}\\
(\forall A \in \wp(S))\, A^{l_* l_*} \subseteq A^{l_*}  \tag{l-Weak Idempotency}\\
(\forall A \in \wp(S))\, A^{u_*} \subseteq A^{u_* u_*}  \tag{u-Weak Idempotency}\\
\emptyset^{l_*} = \emptyset = \emptyset^{u_*}  \tag{Bottom}\\
S^{l_*} = S = S^{u_*}   \tag{Top} 
\end{align}
\end{theorem}

\begin{remark}
The proof of the above theorem is direct. Monotonicity of the approximations need not hold in general. 
This is because the choice of parthood is not \emph{sufficiently coherent} with $\sigma$ in general. A sufficient condition can be that $\sigma$-ideals be generated by $\sigma$ that are at least quasi orders.
\end{remark}

\begin{proposition}
In a GOSI $S$ all of the following hold:
\begin{itemize}
\item {The granulation is not admissible}
\item {The approximations $l_*, u_*$ are not granular}
\end{itemize}
\end{proposition}
\begin{proof}
It is clear that both the lower and upper co-granular approximations are not representable in terms of granules using set operations alone on $\wp(S)$. So weak representability (WRA) fails.
\qed 
\end{proof}

The properties of approximations in a GOSIH depend on those of the ideal and the granulation operator $\gamma$. This is reflected in the next theorem:

\begin{theorem}
In a GOSIH $S$ satisfying
\begin{itemize}
\item {$\sigma$ is supremal,}
\item {$\sigma$ is a quasi order and}
\item {$(\forall a) a\in \gamma (a)$}
\end{itemize}
then all of the following hold:
\begin{align}
A^{l_o} \subseteq A \subseteq A^{u_o};\; \emptyset^{u_o}=\emptyset \,;\, X^{l_o} = X \\
A\subset B\longrightarrow A^{l_o}\subseteq B^{l_o} \,\&\,A^{u_o}\subseteq B^{u_o} \\ 
A^{l_o l_o}= A^{l_o} ; \; A^{u_o} \subseteq A^{u_o u_o} \\
(\forall A\in \mathbb{I})\,  A^{u_o} = A ; \;(\forall A^c\in \mathbb{I})\,  A^{l_o} = A
\end{align}
\end{theorem}
\begin{proof}
Some parts are proved below:
\begin{itemize}
\item {Monotonicity happens because the $\sigma$-ideals behave reasonably well. If $A\subset B$, then for any $z\in A^{l_o}$, it is necessary that $\gamma(z)\cap B^c \subseteq \gamma(z)\cap A^c$. This ensures that $A^{l_o} \subseteq B^{l_o}$. Again for the upper approximation, If $A\subset B$ and $z\in A^{u_o}$, $\gamma(z)\cap A \notin \mathcal{I}$ holds. As then $\gamma(z) \cap A \subseteq \gamma(z)\cap B$, it follows that $\gamma(z)\cap B \notin \mathcal{I}$. This ensures $A^{u_o} \subseteq B^{u_o}$.  }
\item {For proving $(\forall A\in \mathbb{I})\,  A^{u_o} = A$, note that if $z\in A^{u_o}$ then it is necessary that $\gamma (z)\cap A \notin \mathbb{I}$ or $z\in A$. It is not possible that $\gamma (z)\cap A \notin \mathbb{I}$ as $A$ is in $\mathbb{I}$. So $A^{u_o} = A$.}
\item {Again, if $z\in A$ and $\gamma (z)\cap A^c\in \mathbb{I}$ and $A^c\in \mathbb{I}$, then $z\in A$ and $z\in A^{l_o}$. This yields $A= A^{l_o}$. }
\end{itemize} 
\end{proof}

\emph{All this shows that the ideal based approach works due to the properties of the ideals}.
\begin{definition}
For all of the above cases, a natural concept of $A$ being \emph{roughly included} in $B$ ($A\sqsubseteq B$) if and only if $A^l \subseteq B^l$ and $A^u \subseteq B^u$ for relevant choices of $l, u$.   $A$ is roughly equal to $B$ ($A\approx_{l,u} B$) if and only if $A\sqsubseteq B \, \& B\sqsubseteq A$. Quotients of the equivalence $\approx_{l,u}$ will be referred to as rough objects.
\end{definition}

\begin{theorem}
The antichain based semantics of \cite{AM9114,AM6999} applies to all of  GOSI and GOSIH contexts with corresponding concepts of rough equalities.   
\end{theorem}

The proof consists in adapting the entire process of the semantics in the papers to the context. Granularity of approximations is not essential for this. More details will appear in a separate paper.

\section{Meanings of Generalization and Parallel Rough Universes}

The generalizations of the classical definition of approximations derived from approximation spaces and general approximation spaces using ideals differ substantially from those that have been introduced in this research paper. Both  possible meanings and properties differ substantially in a perspective that is formulated below.  An extended example is also used to illustrate some of the concepts introduced in this paper.

In all algebras of arbitrary finite type, ideals can be viewed in second order perspectives as subsets satisfying closure and absorption conditions. Ideals can be described through first order conditions \cite{AP1984} in some universal algebras with distinguished element $0$ (that are \emph{ideal determined}). In the latter case every ideal is always the $0$-class of a congruence. In both cases, ideals behave like higher order zeros. 

The following implication holds always in a GOSI for a point $x\in S$, $A\subseteq S$ and $K$ being an ideal: 
\begin{equation}
\gamma x \subseteq A \longrightarrow \gamma x \cap A^c =\emptyset  \longrightarrow  \gamma x \cap A^c \subseteq K 
\end{equation}
The same idea essentially extends to GOSIH where it is also possible to regard subsets of ideals as essential zeros.

But a \emph{subset of an ideal need not behave like a generalized zero in general}. This statement is opposed to \emph{the subset is part of a generalized zero}. But if every subset is contained in a minimal ideal, the restriction becomes redundant. In the absence of any contamination avoidance related impositions, the latter statement is justified only under additional conditions like \emph{the subset is part of a specific generalized zero}. All this is behind the motivation for the  definition of GOSIH. In the relational approximation contexts of \cite{AT2011,ABA2008,KYG2016}, subsets of generalized zeros are generalized zeros. Apart from wide differences in properties, major differences exist on the nature of ideals. Therefore \emph{if a property has little to do (in a structured way) with what something is not, then that something has the property in an approximate sense}. The idea of \emph{little to do with} or \emph{set no value of} relative operations is intended to be captured by concepts of ideals. 

\emph{If $\sigma$-ideals are seen as essentially empty sets, then they have a hierarchy of their own and function like definite entities. The $\sigma$-ideals under some weak conditions permit the following association. If $A$ is a subset then it is included in the smallest $\sigma$-ideal containing it and a set of maximal $\sigma$-ideals contained in it. These may be seen as a representation of rough objects of a parallel universe}. 

This motivates the following definition:

\begin{definition}
In a GOSI $S =\left\langle \underline{S}, \sigma, \mathcal{G}, l , u\right\rangle$ in which $\sigma$ is supremal every subset $A\subseteq S$ can be associated with a set of maximal $\sigma$-ideals ($\mu(A)$) contained in $S$ and least $\sigma$-ideal $\Upsilon (A) $ containing it. These will be termed parallel rough approximations and pairs of the form $(a, b)$ (with $a\in \mu (A)$ and $b= \Upsilon(A)$) will be referred to as \emph{parallel rough objects}. Elements of $\mu(A)$ will be referred to as \emph{lower parallel approximation} and $\Upsilon (A)$ as the \emph{upper parallel approximation}.
\end{definition}

The lower parallel approximations can be useful for improving the concept of GOSI with additional approximations because they refer to inclusion of worthless things. This is done next.

\begin{definition}
In a GOSI $S$ with supremal $\sigma$, for any subset $A\subseteq S $, the \emph{strong lower and strong upper approximations} will be as follows:
\begin{align}
 A^{l_s} = \{x \,:\, x\in A \, \&\, \{\emptyset\} \subset \mu(\gamma(x) \cap A^c) \}  \tag{s-lower}\\
 A^{u_s} = \{a : a \in S\, \&\,\gamma(a)\cap A \subset \Upsilon (\gamma(a)\cap A)\} \cup A \tag{s-upper}
\end{align}
\end{definition}

\begin{proposition}
In a GOSI $S$, for any $A\subseteq S$, the following hold:
\begin{align*}
A^{l_*}\subseteq A^{l_s} \subseteq A \\
A^{u_*} = A^{u_s}
\end{align*}
\end{proposition}
\begin{proof}
If $\{\emptyset\} \subset \mu(\gamma(x) \cap A^c)$, then there is at least one nonempty $\sigma$-ideal included in $\gamma(x) \cap A^c$. This does not imply that $\gamma(x) \cap A^c$ is a $\sigma$-ideal, but the converse holds. So 
$A^{l_*}\subseteq A^{l_s} \subseteq A$ follows.

For the second part, as $\sigma$ is supremal, $\gamma(a)\cap A \subset \Upsilon (\gamma(a)\cap A)$ ensures that $\gamma(a)\cap A$ is not a $\sigma$-ideal. The converse is also true. $A^{u_*} = A^{u_s}$ follows from this.
\end{proof}

Often it can happen that objects/entities possessing some set of properties are not favored by objects/entities having some other set of properties. This meta phenomena suggests that anti chains on the collection of $\sigma$-ideals can help in associated exclusions and inclusions.

\begin{definition}
 In a GOSI $S$, let $\mathcal{O}$ be an antichain in $\mathcal{I}_{\sigma}(S)$ and $\mathcal{O}^+ =\{B : B\in \mathcal{I}_\sigma (S) \,\&\, (\exists C\in \mathcal{O}) C\subseteq B \}$, $\mathcal{O}^- = \mathcal{I}_\sigma (S) \setminus \mathcal{O}^+$. For any subset $X\subseteq S $, the \emph{a-lower and a-upper approximations} will be as follows:
\begin{align*}
X^{l_a} = \{a : a \in X\, \&\, \gamma(a) \cap X^c \in \mathcal{O}^-\} \tag{a-Lower}\\
X^{u_a} = \{a : a \in S\, \&\,\gamma(a)\cap X \notin \mathcal{O}^-\} \cup X \tag{a-Upper}\\
\end{align*}
The resulting GOSIS of this form will be referred to as a \emph{GOSIS induced by the antichain} $\mathcal{O}$.
\end{definition}

\begin{proposition}
 In a GOSI $S$, for any $X\subseteq S$ and nontrivial antichain $\mathcal{A}$, the following hold:
\begin{align*}
X^{l_a}\subseteq X^{l_*} \subseteq X \\
X^{u_*} \subseteq X^{u_a}
\end{align*}
\end{proposition}

\subsection{Abstract Examples}
{Let } $S= \{a, b, c, e, f, g \} $, and 
$\sigma = \{(a, c),\,(a, e),\,(b, c), \,(b,e ),\,(c,c ),\,(c,b )$,\\ $(e, a),\,(f,f )\}$ be a binary relation on it. It is a not symmetric, transitive or reflexive. In Table \ref{table1}, the computed values of the set of upper bounds, lower bounds, and neighborhoods are presented. $*$ in the last two rows refers to any element from the subset $\{a, b, c, e \}$. Values of the form $U(x, x) $ and $L(x, x)$ have been kept in Table \ref{table2} because they correspond to values of neighborhoods of $\sigma$. 

\begin{table}[hbt]
 \centering
\begin{tabular}{lccc}
\toprule
\textsf{Pair} $(x, z)\;$ & $U(x,z)$  & $L(x, z)$   \\
\midrule
$(a ,b )$ & $\{e, c \}$ & $\emptyset$\\
\midrule
$(a ,c )$ & $\{c \}$ & $\emptyset$\\
\midrule
$(a ,e )$ & $\emptyset$ & $\emptyset$\\
\midrule
$(b ,c )$ & $\{c \}$ & $\{c \}$\\
\midrule
$( b,e )$ & $\emptyset$ & $\emptyset$\\
\midrule
$(c ,e )$ & $\emptyset$ & $\{a,b \}$\\
\midrule
$(* ,f )$ & $\emptyset$ & $\emptyset$\\
\midrule
$(* ,g )$ & $\emptyset$ & $\emptyset$\\
\bottomrule
\end{tabular}

\caption{Upper and Lower Bounds}\label{table1}
\end{table}

In Table \ref{table2}, $\mathbf{U(x,x)=[x]^\sigma}$ and $\mathbf{L(x, x)=[x]_\sigma}$.

\begin{table}[hbt]
 \centering
\begin{tabular}{lcccc}
\toprule
 $\mathbf{x} \;$ & $\;\mathbf{U(x,x)}\;$  & $\;\mathbf{L(x, x)}\;$ & $\;\mathbf{<x>}\;$   \\
\midrule
$a $ & $\{ c, e\}$ & $\{e \}$ & $\{a \}$\\
\midrule
$ b$ & $\{c, e \}$ & $\{c \}$ & $\{ b. c\}$\\
\midrule
$ c$ & $\{ b, c\}$ & $\{a, b, c \}$ & $\{c \}$\\
\midrule
$ e$ & $\{c \}$ & $\{ a, b\}$ & $\{ c, e\}$\\
\midrule
$f $ & $\{ f\}$ & $\{ f\}$ & $\{f \}$\\
\midrule
$ g$ & $\emptyset$ & $\emptyset$ & $\emptyset$\\
\midrule
\bottomrule
\end{tabular}

\caption{Neighborhoods}\label{table2}
\end{table}

Given the above information, it can be deduced that 

\begin{proposition}
In the context, the nontrivial $\sigma $- ideals are $I_1 = \{a, b, e, c \}$ and $I_2 = \{a, b, e, c, f\}$
\end{proposition}

If a co-granulation is defined as per $\gamma(a) = \{b\},  \gamma(b)= \{g\}, \, \gamma(c) = \{c, a\},$\\  $ \gamma(e) = \{e\}, \, \gamma(f) = \{f\}, \, \gamma(g) = \{g, b, c\}$, then the GOSI approximations of the set $A= \{a, b\}$ can be computed to be {$A^l_* = \{b\}$} and {$A^u_* = \{a, b, c, g\}$}. For distinct lattice ideals many approximations of $A$ by $l_k$ and $u_k$ can be computed.

GOSIH related computations of approximations are bound to be cumbersome even for four element sets and so have been omitted.

\subsection{Example: On Dating}

Dating contexts can involve a huge number of variables and features. Expression of these depend substantially on the level of inclusion of diverse genders and sexualities in the actual context. A person's choice of pool of potential dates depends on factors including the person's sexuality. 

People may decide on their potential dating pool by excluding parts of the whole pool and focusing on specific subsets. The operation of \emph{excluding parts of the pool} often happens as a multi stage process involving progressive additions to desired features or confirmation of undesired features. This means that a person's construction of relative dating pools must be happening through rough approximations based on ideals. An ideal can include a number of features, but in general it can be impossible to collect all undesired features in a single ideal as then it would not correspond to anything remotely actualizable.

Typically actualizability depends on the reasoning strategies adopted by the person in question. It is not that everybody thinks in terms of concrete people with undesirable features and people with analogous features - abstraction can be in terms of feature sets.  For example, some lesbian women prefer femme women. But concepts of femme and variants are very subjective in nature. Instances of such classification (or actualization) may depend on exclusion of many features like \emph{muscular build}, but some features like \emph{tattoos} may be desirable/optional. These kind of features may be in the general ideal in question.

\section{Mereotopology and Approximations}

In spatial mereology, spatial regions are associated with elements of a distributive lattice or a Boolean algebra and which in turn are intended to represent collections of regions with operations of aggregation and commonality. Over these binary contact relations $C$ can be defined over them to represent instances of two regions sharing at least one point. Various constructions in the subfield are suited for the ideal based approach to rough approximations. In this section, some of the basic aspects and recent results are stated and connections with approximations are  established. 
All this can be viewed as a new example for the theories invented/developed.

Some concepts and recent results of spatial mereology are mentioned first (see \cite{GDV06}).
\begin{definition}
A \emph{contact relation} $C$ over a bounded distributive lattice $L$ is a binary relation that satisfies
\begin{align*}
Cab \longrightarrow 0 < a \, \&\, 0 < b \tag{C1}\\
Cab \longrightarrow Cba \tag{C2}\\
Cab \,\& \, b\leq e \longrightarrow Cae  \tag{C3}\\
Ca (b\vee e) \longrightarrow Cab \text{ or } Cae \tag{C4}\\
0 < a\wedge b \longrightarrow Cab    \tag{C5}
\end{align*}
If $L$ is a Boolean algebra, then $(L, C)$ is said to be a \emph{contact algebra}. If $C$ satisfies $C1-C4$ alone, then it is said to be a precontact relation and then $(L, C)$ would be a \emph{precontact algebra}.
\end{definition}

C1 is also written as $Cab \longrightarrow Ea \, \&\, Eb$ (for contact implies existence). C5 is basically the statement that overlap implies contact $\oc ab \longrightarrow Cab$. The axioms yield $C(a\vee b) e \longleftrightarrow Cae \text{ or } Cbe $ and $Cab \,\&\, a\leq u \,\&\, b \leq v \longrightarrow C uv $.

When temporal aspects of variation of $C$ are permitted, the predicate for ontological existence $E$ that is defined via $Ea \text{ if and only if } 0 < a $ is too strong as non existence is equated with emptiness. In \cite{DV2017}, to handle variation of existence over time, concepts of \emph{actual existence, actual part of,  actual overlap and actual contact} have been proposed and developed. The actual existence predicate $AE$ is one that satisfies 
\begin{align*}
 AE(1) \,\&\, \neg AE(0) \tag{AE1}\\
 AE(a) \, \&\, a\leq b \longrightarrow AE(b) \tag{AE2}\\
 AE(a\vee b) \longrightarrow AE(a) \text{ or } AE(b) \tag{AE3}
\end{align*}

Subsets of a Boolean algebra that satisfy AE1, AE2 and AE3 are called \emph{Grills}. Every grill is a union of ultrafilters or an ultra filter.

A \emph{discrete space with actual points} is a pair $Z= (X, X^a)$ with $X$ being a nonempty set and $\emptyset \subset X^a \subset X $. For $H\subset X$, let $AE_{Z}(H)$ if and only if $H \cap X^a \neq \emptyset$. If $B(Z)$ is the Boolean algebra of all subsets over $X$, then $(B(Z), AE_{Z})$ would be a Boolean algebra with a predicate of actual existence. It is proved in \cite{DV2017} that 
\begin{theorem}
In Boolean algebras with predicate of actual existence $(B, AE)$, there exist a discrete space $Z= (X, X^a)$ and an isomorphic embedding $h: (B(Z), AE_Z)\longmapsto (B, AE)$. 
\end{theorem}

On a Boolean algebra with an extra predicate for actual existence, it is possible to define the \emph{actual contact} predicate or define the latter as a predicate $C^a$ that satisfies the following axioms:
\begin{align*}
C^{a} 11\,\&\,  C^{a}00\tag{Ca1}\\ 
C^{a}xb \longrightarrow   C^{a}bx \tag{Ca2}\\ 
C^{a} xb \longrightarrow C^{a}xx\tag{Ca3}\\ 
C^{a} xb\, \&\, b\leq z  \longrightarrow C^{a} xz\tag{Ca4}\\ 
C^{a} x(b\vee e) \longrightarrow C^{a} xb \text{ or } C^a xe  \tag{Ca5}
\end{align*}
It is also possible to define a unary predicate $AC$ via $AC(x) \text{ if and only if } C^a xx$.

A subset $H$ of a contact algebra $B $ is a \emph{clan} if it is a grill that satisfies \textsf{CL}: $(\forall a, b\in H)\, Cab$, while a subset $H$ of a precontact algebra $B$ is an \emph{actual clan} if it is a grill that satisfies \textsf{ACL}: $(\forall a, b\in H)\, C^a ab$.

Associated collections of all clans and actual clans will respectively be denoted by $CL(B)$ and $CL^a(B)$ respectively. It can be shown that $CL^a(B) \subseteq CL(B)$ in general.
If $Z$ is the set of all ultrafilters $\mathcal{F}_u(B)$ and $R^a$ is the canonical relation for $C^a$ defined by 
\[R^a UV \text{ if and only if } (\forall x\in U)(\forall b \in V)\, C^a xb\]
The canonical relation $R$ for $C$ is defined in the same way. An ultra filter $U$ is \emph{reflexive} if $R^a UU$ and all reflexive ultrafilters are actual clans. $R^a$ is a nonempty, symmetric and quasi-reflexive relation while $R$ is a reflexive and symmetric relation (Quasi reflexivity is $R^a be \longrightarrow R^a bb$).
All of the following are proved in \cite{DV2017}:
\begin{proposition}
\begin{itemize}
\item {$RUV$ if and only if $R^a UV $ or $U=V$.}
\item {Every clan (resp. actual clan) is a union of nonempty sets of mutually $R$-related (resp. $R^a$-related) ultrafilters.}
\item {All ultrafilters contained in an actual clan are reflexive ultrafilters while any ultra filter contained in a clan is an actual clan or a non-reflexive ultra filter.}
\item {$C^a be$ if and only if $(\exists G\in CL_a(B))\, b, e \in G $.}
\end{itemize} 
\end{proposition}

\subsection{New Approximations}

All of the definitions and results in this subsection are new and differ fundamentally from the approach in \cite{LP2011}. If $(X, X^a)$ is the discrete space associated  with a Boolean algebra with actual contact $(B, C^a)$ and $\gamma : X \longmapsto  \wp(X^a)$ is a map then approximations in $B$ can be defined in at least two different ways (for a fixed actual clan $K\in CL_a(B)$):
\begin{align*}
A^{l_a} = \{x: x\in A \, \&\, \gamma (x)\cap A^c \notin K) \}         \tag{CG-Lower} \\
A^{u_a} = \{x: x\in X \, \&\, \gamma (x)\cap A \in K \} \cup A         \tag{CG-Upper}\\
A^{l_g} = \bigcup\{\gamma(x): \gamma (x)\cap A^c \notin K \}\cap A         \tag{G-Lower} \\
A^{u_g} = \bigcup \{\gamma(x): \gamma (x)\cap A \in K \} \cup A         \tag{G-Upper}\\
A^{l_c} = \bigcup\{H: H\cap A^c \notin K \}\cap A         \tag{Clan-Lower} \\
A^{u_c} = \bigcup \{H: H \in K \,\&\, H\cap A \neq \emptyset \} \cup A         \tag{Clan-Upper}\\
\end{align*}

The properties of these approximations depend to a substantial extent on the definition of $\gamma$ used. One 
possibility is to use the actual-contact relation or a derived mereotopological relation. In the present author's view some  meaningful phrases are \emph{the things in actual contact with, the things in contact with, the most common things that become in actual contact by, the things that become in actual contact by, and the relative wholes determined by}. The first of these can be attempted with the neighborhoods generated by $C^a$ itself in the absence of additional information about the context.

One practical context where approximations of the kind can be relevant is in the study of handwriting of people. Many kinds of variations in the handwriting of people (especially of morphological subunits and their relative placement) can be found over time, location and media used.  

So if $\gamma$ is defined as per $(\forall x)\, \gamma(x) = [x]_{C^a} =  \{b: C^a bx \} $, then the context becomes a specific instance of a GOSI in which the ideals are also regulated by $C$. If instead $\gamma(x) = <x>$ holds, then it is provable that:

\begin{theorem}
All of the following hold (in the context of this subsection) for any two elements of the Boolean algebra with actual contact when $\gamma (x) = <x>$
\begin{align}
A^{l_a} \subseteq A \subseteq A^{u_a}; \; \emptyset^{u_a}=\emptyset = \emptyset^{l_a}\,;\, X^{l_a} = X = X^{u_a} \\
A\subset B\longrightarrow A^{l_a}\subseteq B^{l_a} \,\&\,A^{u_a}\subseteq B^{u_a} \\ 
A^{l_a l_a}= A^{l_a} ; \; A^{u_a u_a} = A^{u_a}\\
(A\cap B)^{l_a} = A^{l_a}\cap B^{l_a} ; \; (A\cup B)^{l_a}\supseteq A^{l_a}\cup B^{l_a}\\
(A\cup B)^{u_a} = A^{u_a}\cup B^{u_a} ; \; (A\cap B)^{u_a} \subseteq A^{u_a}\cap B^{u_a}
\end{align}
\end{theorem}
\begin{proof}
Note that actual clans determine specific subclasses of ideals. 
So all of the above properties follow from results of Section \ref{sei} 
\end{proof}

\begin{theorem}
All of the following hold (in the context of this subsection) for any two elements of the Boolean algebra with actual contact when $\gamma (x) = [x]_{C^{a}} = [x] \text{ (for short)}$
\begin{align}
A^{l_a} \subseteq A \subseteq A^{u_a}\\
\emptyset^{u_a}=\emptyset = \emptyset^{l_a}\,;\, X^{l_a} = X = X^{u_a} \\
A\subset B\longrightarrow A^{l_a}\subseteq B^{l_a} \,\&\,A^{u_a}\subseteq B^{u_a} 
\end{align}
\end{theorem}

\begin{proof}
\begin{itemize}
\item {$A^{l_a} \subseteq A \subseteq A^{u_a}$ follows from definition.}
\item {$X^{l_a} = \{b: b\in X \,[b]\cap \emptyset \notin K \} = X$}
\item {If $A\subset B$, then $(\forall b\in A^{l_a})\, [b]\cap A^c \in K^c$ and $[b]\cap B^c \subseteq [b]\cap A^c$. Since $K^c$ is an ideal, it follows that $A^{l_a}\subseteq B^{l_a}$. 

If $b\in A^{u_a}$, then $[b]\cap A\in K$. Also $[b]\cap A \subseteq [b]\cap B$. But $K$ is a union of ultrafilters, so $[b]\cap B\in K$ and consequently $b\in B^{u_a}$. }
\end{itemize} 
\qed 
\end{proof}

Since the basic duality theorems for actual contact algebras (and contact algebras in particular) are in place\cite{DV2017}, the duality/inverse problem of such algebras enhanced with approximation operators may be solvable with ease. In the present author's opinion, the following formalism would be optimal:

\begin{problem}[Inverse Problem]
Given an algebraic system of the form\\  $A = \left\langle B, C^a, l_a, u_a \right\rangle$ with ${B, C_a}$ being an actual contact algebra and $l_a, u_a$ are unary operations satisfying:
\begin{align}
(\forall z)\, z^{l_a} \subseteq z \subseteq z^{u_a}\\
0^{u_a}=0 = 0^{l_a}\,;\, 1^{l_a} = 1 = 1^{u_a} \\
(\forall z, v) \, (z\subset v\longrightarrow a^{l_a}\subseteq v^{l_a} \,\&\,z^{u_a}\subseteq v^{u_a} \\ 
(\forall z)\,  z^{l_a l_a}= z^{l_a}\\
(\forall z) \,z^{u_a} \subseteq z^{u_a u_a}
\end{align}
under what additional conditions does there exist a neighborhood operator $\gamma$ and an actual clan $K$ that permit a definition of the operators $l_a, u_a$ according to
\begin{align}
z^{l_a} = \{x: x\in z \, \&\, \gamma (x)\cap z^c \notin K) \}         \tag{CG-Lower} \\
z^{u_a} = \{x: x\in B \, \&\, \gamma (x)\cap z \in K \} \cup z         \tag{CG-Upper}
\end{align}
\end{problem}

\section*{Further Directions and Remarks}

In this research, a relatively less explored area in the construction of point wise approximations by ideals has been investigated from new perspectives by the present author. The previously available theory has been streamlined and the  meaning of approximations in the approach has been explained. A concept of co-granular approximations has been introduced to explain the generation of related approximations including the popular point wise rough approximations. Further 
\begin{itemize}
\item {the methodology is generalized to specific modifications of granular operator spaces \cite{AM6999,AM9114} (called co-granular operator spaces) and in particular to lattices generated by collections of sets and lattice ideals,}
\item {the restrictions to general approximation spaces are relaxed, }
\item {knowledge interpretation in the contexts are proposed.}
\item {few meaningful examples and application areas have been proposed,}
\item {ideal based rough approximations are shown to be natural in spatial mereological contexts and }
\item {related inverse problems are posed.}
\end{itemize}
In a forthcoming paper, the fine details of the mentioned antichain based semantics and other algebraic semantics will be considered by the present author. 
\bibliographystyle{splncs.bst}
\bibliography{../bib/biblioam2016xx.bib}
\end{document}